\documentclass[leqno]{amsart}
\usepackage{latexsym,amsthm,amscd,amssymb,amsmath,amsfonts,graphics}
\DeclareRobustCommand\ensuremath[1]{\ifmmode#1\else{$#1$}\fi}
\usepackage[all]{xy}
\usepackage[mathscr]{euscript}
\usepackage{t1enc}
\usepackage[english]{babel}
\newtheorem{theorem}{Theorem}[section]
\newtheorem{lemma}[theorem]{Lemma}
\newtheorem{corollary}[theorem]{Corollary}

\newtheorem{example}[theorem]{Example}
\theoremstyle{definition}

\theoremstyle{remark}
\newtheorem{remark}[theorem]{Remark}

\def\id{\mathop{\rm id}\nolimits}

\newcommand{\injj}{\varinjlim}

\newcommand{\DI}{C^{\ast}(\Delta_1)}
\newcommand{\DII}{C^{\ast}(\Delta_2)}

\newcommand{\CG}{C^{\ast}(\Gamma)}
\newcommand{\CCG}[1]{C^{\ast}(\Gamma_{#1})}
\newcommand{\UDI}{\mathcal{U}(C^{\ast}(\Delta_1))}
\newcommand{\UDIi}{\mathcal{U}(C^{\ast}(\Delta_i))}
\newcommand{\UDII}{\mathcal{U}(C^{\ast}(\Delta_2))}

\newcommand{\UGI}{\mathcal{U}(C^{\ast}(\Gamma_1))}
\newcommand{\UGII}{\mathcal{U}(C^{\ast}(\Gamma_2))}
\newcommand{\UG}{\mathcal{U}(C^{\ast}(\Gamma))}

\newcommand{\C}{{\mathbb C}}

\newcommand{\A}{{\mathscr A}}
\newcommand{\UA}{\mathcal{U}(\mathscr{A})}
\newcommand{\U}{\mathcal{U}}
\newcommand{\Z}{{\mathbb Z}}

\newcommand{\Q}{{\mathbb Q}}
\newcommand{\R}{{\mathbb R}}
\newcommand{\T}{{\mathbb T\,}}
\newcommand{\torus}{\T}

\newcommand{\G}{\Gamma}
\newcommand{\Lwedge}{\text{\Large $\wedge$}}
\begin{document}
\begin{abstract}
    We study to what extent      group $C^\ast$-algebras are
    characterized by their  unitary groups. A complete characterization
    of which  Abelian group $C^\ast$-algebras have isomorphic unitary
    groups is obtained. We compare these results with other
    unitary-related invariants of $C^\ast(\Gamma)$, such as the
    $K$-theoretic $K_1(C^\ast(\Gamma))$ and find that
    $C^\ast$-algebras of nonisomorphic torsion-free Abelian groups may have
    isomorphic $K_1$-groups,  in sharp contrast with the well-known fact
    that $C^\ast(\Gamma)$ (even $\Gamma$) is characterized by the
    topological group structure of its unitary group when $\Gamma $ is
    torsion-free and Abelian.
\end{abstract}
\date{\today}
\title[Characterizing  group $C^\ast$-algebras through their unitary groups]
{Characterizing  group $C^\ast$-algebras through their unitary
groups: the Abelian case}\thanks{Research  partly supported by the
Spanish Ministry of Science (including FEDER funds), grant
MTM2004-07665-C02-01 and Generalitat Valenciana GV/2007/012.}
\author{Jorge Galindo and Ana Mar{\'\i}a  R{\'o}denas}
\address{Departamento de Matem\'aticas \\
         Universidad Jaume I \\
        Campus Riu Sec, 12071\\
        Castell\'on\\
        Spain\\}
        \email{jgalindo@mat.uji.es}
       \hfill \email{arodenas@mat.uji.es}
 \maketitle
\section{Introduction}
The index theorem states that  every continuous $f\colon \T\to \T$
is homotopic to the function $t\mapsto t^n$ for some $n\in \Z$ (its
winding number). As a consequence the quotient of the unitary group
of $C^\ast(\Z)$ by its connected component  is isomorphic to $\Z$.
This identification can be extended in a functorial fashion to
finitely generated Abelian groups and their  inductive limits. Since
every torsion-free Abelian group is an inductive limit of finitely
generated groups,  the following  theorem, that we take as the
departing point of our paper, follows.
\begin{theorem}[see Theorem 8.57  of \cite{hoffmorr}]\label{index} If $\Gamma$ is a
torsion-free Abelian group the quotient $\mathcal{U}/\mathcal{U}_0$
of the unitary group $\mathcal{U}=\mathcal{U}(C^\ast(\Gamma))$ by
its connected component $\mathcal{U}_0$ is isomorphic to $\Gamma$.
Hence, two torsion-free Abelian groups $\Gamma_1$ and $\Gamma_2$
with topologically isomorphic unitary groups
$\mathcal{U}(C^\ast(\Gamma_1))$ and $\mathcal{U}(C^\ast(\Gamma_2))$
must already be isomorphic.
\end{theorem}

Theorem \ref{index} suggests the usage of $\UG$ as an invariant for
$\CG$. To determine its strength it is necessary to know  to what
extent the topological group structure of $\UG$ determines  $\CG$.
As a first step in this direction, we devote Section \ref{charac} to
characterize when two Abelian groups $\Gamma_1$ and $\Gamma_2$ have
isomorphic unitary groups.  The groups
$\mathcal{U}(C^\ast(\Gamma_1))$ and $\mathcal{U}(C^\ast(\Gamma_2))$
are shown to be topologically isomorphic if and only if
$|\Gamma_1/t(\Gamma_1)|=|\Gamma_2/t(\Gamma_2)|=:\alpha$ and
$\oplus_\alpha\Gamma_1/t(\Gamma_1)$ is group-isomorphic to
$\oplus_\alpha\Gamma_2/t(\Gamma_2)$, where $t(\Gamma_i)$ stands for
the torsion subgroup of $\Gamma_i$.

 Another unitary-related invariant of
$C^\ast(\Gamma)$ of great importance is the $K_1$-group,
$K_1(C^\ast(\Gamma))$. Since $K_1(C^\ast(\Z^m))=\Z^{2^{m-1}}$, two
torsion-free finitely  generated Abelian groups are isomorphic
whenever their $K_1$-groups are. The way this fact is proved does
not however allow a functorial extension to inductive limits and,
indeed, we construct in Section \ref{sec:counter}  two nonisomorphic
torsion-free Abelian groups $\Gamma_1$ and $\Gamma_2$ with
isomorphic $K_1$-groups, thereby showing that Theorem \ref{index} is
not valid for $K_1$-groups instead of unitary groups.  We find
therefore that $\mathcal{U}(C^\ast(\Gamma))$ is a stronger invariant
than $K_1(C^\ast(\Gamma))$, for torsion-free Abelian groups. For
general (even Abelian) groups this is no longer true,
$K_1(C^\ast(\Gamma))$ distinguishes between groups with different
finitely generated torsion-free quotients, while $\UG$ need not, see
Section \ref{concl}.


\section{Background}
This paper is  concerned with group $C^\ast$-algebras. The
$C^\ast$-algebra $C^\ast(\Gamma)$  of a group $\Gamma$ is defined as
the enveloping $C^{\ast}$-algebra of the convolution algebra
$L^1(\Gamma)$ and, as such, encodes the representation theory of
$\Gamma$,  see \cite[Paragraph 13]{dixm}.

We analyze   in this paper  to what extent a group $\Gamma$, or
rather the $C^\ast$-algebra structure of $C^\ast(\Gamma)$, is
determined by the topological  group structure of
$\U(C^\ast(\Gamma)\,)$.

The unitary groups $\U(C^\ast(\Gamma))$ are obviously related to
another invariant of $C^\ast(\Gamma)$ of greater importance, the
$K_1$-group of $K$-theory. $K$-theory for $C^{\ast}$-algebras is
based on two functors, namely, $K_0$ and $K_1$, which associate to
every $C^\ast$-algebra $\A$, two Abelian groups $K_0(\A)$ and
$K_1(\A)$. The group  $K_1(\A)$ is in particular defined by
identifying unitary elements of \emph{matrix algebras} over $\A$. It
is allowing matrices over  $\A$ (instead of elements of $\A$) that
makes $K_1$-groups Abelian. If $\Gamma$ is a discrete group, there
is a natural embedding of $\Gamma$ in $\UG$, this may be composed
with the canonical map $\omega \colon \UG\to K_1(\CG)$. $K_1(\CG)$
being Abelian, the resulting homomorphism factors through the
Abelianization of $\Gamma$, yielding a homomorphism   $\kappa_\Gamma
\colon \Gamma/\Gamma^\prime \to K_1(\CG)$. $\kappa_\G$  was shown to
be rationally injective in \cite{ellinats}, see also
\cite{bettvale}.

Now and for the rest of the paper we restrict  our attention to
discrete Abelian groups.  When $\Gamma$ is a discrete Abelian group,
$C^\ast(\Gamma)$ is a commutative $C^\ast$-algebra with spectrum
homeomorphic to the compact group $\widehat{\Gamma}$, the group of
characters of $\Gamma$. We may thus identify $C^\ast(\Gamma)$ with
the algebra of continuous functions
 $C(\widehat{\Gamma},\C)$ and the   Gelfand transform
 coincides with the Fourier transform. The unitary group
 $\U(C^\ast(\Gamma))$ can therefore be identified with the
 topological group of $\T$-valued functions
 $C(\widehat{\Gamma},\T)$. Hence relating $\UG$ to $\G$     amounts  in this case to
relating   $\Gamma$ to  $C(\widehat{\Gamma},\T)$.

Also, for commutative $\A$ (as is the case with $\CG$, with $\G$
Abelian), the determinant map $\Delta \colon K_1(\A)\to \UA/\UA_0$
is a \emph{right inverse} of the canonical map $\omega \colon
\UA/\UA_0\to K_1(\A)$ (see \cite[Section 8.3]{rordlarslaus}) and the
link between $K_1(\A)$ and $\U(\A)$ is stronger.

The commonly used  notation $K_\ast(\A)=K_1(\A)\oplus K_0(\A)$ will
also be adopted in this paper.

\textbf{Acknowledgement:} We would like to thank Pierre de la Harpe
for some suggestions and references  that helped us to improve the
exposition  of this paper.
\section{A torsion-free Abelian group $\Gamma$ not determined by
$K_1(C^\ast(\Gamma))$}\label{sec:counter} As stated in the
introduction, there is a group isomorphism $\mathrm{In}\colon
C(\T,\T)/C(\T,\T)_0 \to \Z$ assigning to every $f\in C(\T,\T)$ its
winding number. In other words,  every element of $C(\T,\T)$ is
homotopic to exactly one character of $\T$. This point of view can
be carried over to $\T^n$ and then, taking projective limits,  to
every compact connected Abelian group, ultimately leading to Theorem
\ref{index}, after identifying  $C(\widehat{\Gamma},\T)$  with
$\U(C^\ast(\Gamma))$.

Despite the strong relation between $\U(C^\ast(\Gamma))$ and
$K_1(C^\ast(\Gamma))$ we construct in this section two nonisomorphic
torsion-free Abelian groups $\Gamma_1$ and $\Gamma_2$ with
$K_1(\U(C^\ast(\Gamma_1)))$ isomorphic to
$K_1(\U(C^\ast(\Gamma_2)))$.

\subsection{The structure of $K_1(C^\ast(\Gamma))$ for torsion-free
Abelian $\Gamma$} A countable torsion-free Abelian group $\Gamma$
can always be obtained as the inductive limit  of torsion-free
finitely generated  Abelian groups. Simply enumerate
$\Gamma=\{\gamma_n\colon n<\omega\}$, define $\Gamma_n =\langle
\gamma_j \colon 1\leq j\leq n\rangle$ and let $\phi_n\colon \Gamma_n
\to \Gamma_{n+1}$ define the inclusion mapping, then
$\Gamma=\varinjlim (\Gamma_n,\phi_{n})$. Each homomorphism $\phi_n$
then induces a morphism of $C^\ast$-algebras $\phi_n^\ast \colon
C^\ast(\Gamma_n)\to C^\ast(\Gamma_{n+1})$, and $C^\ast(\Gamma)=
\varinjlim (C^\ast(\Gamma_n),\phi^\ast_{n})$

The functor $K_1$ commutes with inductive limits, see for instance
\cite{rordlarslaus}.  If $K_1(\phi_n)\colon
K_1(C^{\ast}(\Gamma_{n}))\to K_1(C^{\ast}(\Gamma_{n+1}))$ denotes
the homomorphism  induced by the morphism $\phi_n^\ast$ , we have
\[ K_1(C^\ast(\Gamma)) =\injj
(K_1(C^{\ast}(\Gamma_{n})),K_1(\phi_n)).\]

The groups $\Gamma_n$  in the above discussion are all isomorphic to
$\Z^{k(n)}$, for suitable $k(n)$, and it is well-known that
$K_\ast(C^\ast(\Z^k))$ is isomorphic to the exterior product
$\Lwedge \Z^k$. Since this realization of  $K_1(C^\ast(\Gamma))$
through  exterior products will be essential in determining our
examples, we next recall  some basic facts about them.

The $k$-th exterior, or wedge, product $\Lwedge^{k}(\Z^n)$ of a
finitely generated group $\Z^n$ with free generators $e_1,\ldots,
e_n$ is isomorphic to the free Abelian group generated by $\bigl\{
e_{i_1}\Lwedge \cdots \Lwedge e_{i_k}\colon
\{i_1,\ldots,i_k\}\subset \{1,\ldots,n\}\bigr\}$. A group
homomorphism $\phi\colon \Z^n \to \Z^m$  induces a group
homomorphism $\Lwedge^k(\phi)\colon \Lwedge^k(\Z^n)\to
\Lwedge^k(\Z^m)$ in the obvious way
$\Lwedge^k(\phi)(e_{i_1}\Lwedge\cdots \Lwedge
e_{i_k})=\phi(e_{i_1})\Lwedge\cdots \Lwedge \phi(e_{i_k})$. If
$\Gamma=\injj(\Gamma_i,h_i)$ is a direct limit, $\Lwedge^k (\Gamma)$
can be obtained as
$\Lwedge^k(\Gamma)=\injj(\Lwedge^k(\Gamma_i),\Lwedge^k(h_i))$. Other
elementary properties of exterior products are best understood
taking into account that $\Lwedge \Gamma$ is isomorphic to the
quotient of $\bigotimes \Gamma$ by the two-sided ideal    $N$
spanned by tensors of the form $g\otimes g$. The reference
\cite{bouralg} is a classical one concerning exterior products.

The following result is well-known (\cite{bren,elli}), we supply a
proof for the reader's convenience.
\begin{lemma}[\cite{elli}, Paragraph 2.1]\label{exter} Let $\Gamma $ be a
torsion-free discrete Abelian group.
 Then
\[
K_1(C^{\ast}(\Gamma))\cong \Lwedge_{\text{\tiny odd}}\, \Gamma
:=\bigoplus_{j=0}^{\infty} \Lwedge^{2j+1}\Gamma.
\]
\label{kgrupo-ext}
\end{lemma}
\begin{proof}
Recall in first place that there is a unique ring isomorphism
$R\colon \Lwedge \Z^n \to K_\ast(C^\ast(\Z^n))$ respecting the
canonical embeddings of $\Z^n$ in both  $K_\ast(C^\ast(\Z^n))$ and
$\Lwedge \Z^{n}$. Since
$K_\ast(C^\ast(\Z^n))=K_0(C^\ast(\Z^n))\oplus K_1(C^\ast(\Z^n))$ and
the ring structure $K_\ast(C^\ast(\Z^n))$ is $\Z_2$-graded (which
means that $x \in K_i(C^\ast(\Z^n))$, $y \in K_j(C^\ast(\Z^n))$
implies that $xy \in K_{i+j}(C^\ast(\Z^n))$ with $i,j\in \Z_2$), we
have that the isomorphism $R$ maps $\Lwedge_{\text{\tiny odd}}\Z^n$
onto $K_1(C^\ast(\Z^n))$.

Now put $\Gamma=\injj(\Gamma_n,\phi_n)$ with $\Gamma_n\cong
\Z^{j_n}$. The uniqueness of the above mentioned ring-isomorphism,
together with the fact that wedge products commute with direct
limits implies that $K_1(C^\ast(\Gamma))$ is isomorphic to
${\displaystyle  \Lwedge_{\text{odd}} \Gamma}$.
\end{proof}
Since the groups $\Gamma_n$ are always isomorphic to $\Z^{k(n)}$ a
comparison between $\Gamma$ and $K_1(C^\ast(\Gamma))$ turns into a
comparison of two inductive limits, $\injj(\Z^{k(n)},\phi_n)$ and
$\injj(\Z^{2^{k(n)-1}},K_1(\phi_n))$. When $\Gamma$ has finite rank
$m$ it may be assumed without loss of generality that $k(n)=m$ for
all $n$. If in addition $m\leq 2$, it is easy to see (cf. Lemma
\ref{lem:desc}) that $K_1(\phi_n)=\phi_n$. We have thus:
\begin{corollary}\label{rank2}
If $\Gamma$ is a torsion-free Abelian group of rank $\leq 2$, then
$K_1(C^\ast(\Gamma)) $ is isomorphic to $\Gamma$.
\end{corollary}
Corollary \ref{rank2} shows that two nonisomorphic torsion-free
Abelian groups $\Gamma_i$ with $K_1(C^\ast(\Gamma_1))$ isomorphic to
$K_1(C^\ast(\Gamma_2))$ must have rank larger than 2.  For our
counterexample we  will deal with two groups of rank 4. If $\Gamma$
is such a group, then $K_1(C^\ast(\Gamma))$ is isomorphic to
$\Lwedge^1(\Gamma) \oplus \Lwedge^3(\Gamma)$. Our  selection  of the
examples is  determined by the following theorem of Fuchs and
Loonstra.
\begin{theorem}[Particular case of  Theorem 90.3 of \cite{fuchs2}]\label{teo-fuchs}
There are two nonisomorphic groups  $\Gamma_1$ and $\Gamma_2$, both
of rank 2, such that
\[
\Gamma_1\oplus \Gamma_1\cong \Gamma_2\oplus \Gamma_2.\]
\end{theorem}
We then have:
\begin{theorem}\label{counter}
Let $\Gamma_1,\Gamma_2$ be the groups of Theorem \ref{teo-fuchs} and
define the 4-rank groups, $\Delta_i=\Z\oplus \Z \oplus \Gamma_i$.
Then $K_1(C^\ast(\Delta_1))$ and  $K_1(C^\ast(\Delta_2))$ are
isomorphic, while $\Delta_1$ and $\Delta_2$ are not.
\end{theorem}
We shall split the proof of Theorem \ref{counter} in several Lemmas.
We begin by observing how Lemma \ref{exter} makes the groups
$K_1(C^\ast(\Delta_i))$ easily realizable.
\begin{lemma}\label{lem:desc}
If $\Gamma$ is a torsion-free Abelian group of rank 2 and
$\Delta=\Z\oplus \Z \oplus \Gamma$, then
\[ K_1(C^\ast(\Delta)) \cong  \Z\oplus  \Z  \oplus \Gamma \oplus
\Gamma \oplus \Lwedge^2\Gamma\oplus \Lwedge^2\Gamma .\]
\end{lemma}
\begin{proof}
As $\Delta$ has rank 4,
\begin{equation}\label{eq:wed}\Lwedge_{\text{\tiny odd}}\Delta=\Lwedge^1 \Delta
\oplus \Lwedge^3\Delta\cong \Delta \oplus
\Lwedge^3\Delta.\end{equation}

Put $\Gamma =\injj (\Gamma_{n},\phi_{n})$, with $\Gamma_{n} \cong
\Z^2$.  Then, defining $\id\oplus \id \oplus \phi_{n} \colon
\Z\oplus \Z\oplus \Gamma_{n}\to \Z\oplus \Z\oplus \Gamma_{n+1}$ in
the obvious way, we have that $\Delta=\injj(\Z\oplus \Z\oplus
\Gamma_{n},\id\oplus \id \oplus \phi_{n})$ and $\Lwedge^3\Delta
=\injj (\Lwedge^3(\Z\oplus \Z\oplus\Gamma_{n}),\Lwedge^3(\id \oplus
\id \oplus \phi_{n}))$.

If $e_j^n$, $j=1,2$ are the generators of $\Z\oplus \Z$ and $f_j^n$,
$j=1,2$ are the generators of $\Gamma_{n}$, $\Lwedge^3(\Z\oplus
\Z\oplus\Gamma_{n})=\langle e_1^n\Lwedge e_2^n\Lwedge f_1^n
,e_1^n\Lwedge e_2^n\Lwedge f_2^n,e_1^n\Lwedge f_1^n\Lwedge
f_2^n,e_2^n\Lwedge f_1^n\Lwedge f_2^n\rangle$. The images of each of
these generators under the homomorphism $\Lwedge^3(\id\oplus
\id\oplus\phi_{n})$ are:

\begin{align*}
\Lwedge^3(\id\oplus \id\oplus\phi_{n})\biggl(e_1^n\Lwedge
e_2^n\Lwedge
f_j^n\biggr)&= e_1^{n+1}\Lwedge e_2^{n+1}\Lwedge \phi_{n}(f_j^n), \quad {j=1,2}\\
\Lwedge^3(\id\oplus \id\oplus\phi_{n})\biggl(e_j^n\Lwedge
f_1^n\Lwedge f_2^n\biggr)&= e_j^{n+1}\Lwedge \bigl(
\Lwedge^2(\phi_{n})(f_1^n\Lwedge f_2^n)\,\bigr), \quad {j=1,2}.
\end{align*}

In the limit, the thread formed by the first two generators will
yield a copy of $\Gamma$ while the one formed by each of the other
two will yield a copy of $\Lwedge^2\Gamma$. This and \eqref{eq:wed}
give the Lemma.
\end{proof}
 We now take care of $\Lwedge^2(\Gamma)$. This is a rank one
group. Abelian groups of rank one are completely classified by their
so-called type.

The type of an Abelian group $A$ is defined in terms of $p$-heights.
 Given a prime $p$,
the largest integer $k$ such that $p^{k}\mid a$ is called the
\emph{$p$-height} $h_p(a)$ of $a$. The sequence of $p$-heights
$\chi(a)=(h_{p_1}(a),\ldots,h_{p_n}(a),\ldots)$, where $p_1,\ldots,
p_n,\ldots$ is an enumeration of the primes, is then called the
\emph{characteristic} or the \emph{height-sequence} of $a$. Two
characteristics $(k_1,\ldots,k_n,\ldots)$ and
$(l_1\,\ldots,l_n,\ldots)$ are considered equivalent if $k_n= l_n$
for all but a finite number of finite indices. An equivalence class
of characteristics is called a \emph{type}. If $\chi(a)$ belongs to
a type $\mathbf{t}$, then we say that \emph{ a is of type }
$\mathbf{t}$.  In a  torsion-free group of rank one  all elements
are of the same type (such groups are called \emph{homogeneous}).
For more details about $p$-heights and types, see \cite{fuchs2}. The
only fact we need here is that two groups of rank 1 are isomorphic
if and only if they have nontrivial elements with the same type,
Theorem 85.1 of \cite{fuchs2}.

We now study the type of  groups $\Gamma \Lwedge \Gamma$ with
$\Gamma$ of rank 2.
\begin{lemma}\label{det}
Let $\Gamma$ be a torsion-free  group of rank 2 and let $x_1,x_2\in
\Gamma$.  The element $ x_1\Lwedge x_2\in \Gamma \Lwedge \Gamma$ is
divisible by the integer $m$ if and only if there is  some element
$k_1x_1+k_2 x_2 \in \Gamma $  divisible by $m$  with either   $k_1$
or $k_2$ coprime with  $m$.
\end{lemma}
\begin{proof}
We can without loss of generality assume that the subgroup generated
by $x_1,x_2$ is isomorphic to $\Z^2$ and that $\Gamma$ is an
additive subgroup of the vector space  spanned  over $\Q$ by
$x_1,x_2$. Now $x\Lwedge y$ will be divisible by $m$ if and only if
there are elements $u_1,u_2$ in $\Gamma$ such that
$u_i=\alpha_{i1}x_1+\alpha_{i2}x_2$ with $ \det (\alpha_{ij})=1/m$
(note that $u_1\Lwedge u_2= \det(\alpha_{ij})\,x_1\Lwedge x_2$). To
get that determinant we clearly need some denominator $m$ and we can
assume (by conveniently modifying the $\alpha_{ij}$'s) that
$\alpha_{11}=k_{1}/m$ and $\alpha_{12}=k_{2}/m$ with either $k_1$ or
$k_2$ coprime with $m$. The element of $\Gamma$ we were seeking is
then $k_1x_1+k_2x_2$.
 \end{proof}
\begin{lemma}Let $\Gamma_1$ and $\Gamma_2$ be two rank 2, torsion-free Abelian
groups.  If   $\,\Gamma_1\oplus \Gamma_1\cong \Gamma_2\oplus
\Gamma_2$, then $\Lwedge^{2}(\Gamma_1)\cong
\Lwedge^{2}(\Gamma_2)$.\label{segprodext}
\end{lemma}
\begin{proof}
Let $\{v_1,w_1\}$ and $\{v_2,w_2\}$ be maximal independent sets in
$\Gamma_1$ and $\Gamma_2$, respectively and denote by $\phi\colon
\Gamma_1\oplus \Gamma_1 \to \Gamma_2\oplus \Gamma_2$ the
hypothesized isomorphism.

 By conveniently re-defining the
elements $v_i$ and $w_i$ it may be assumed that
\begin{align*}
\phi(v_1,0)&= (\alpha_{11}v_2+\alpha_{12}w_2,\beta_{11}
v_2+\beta_{12} w_2)\\
\phi(w_1,0)&=(\alpha_{21}v_2+\alpha_{22}w_2,\beta_{21}
v_2+\beta_{22} w_2),
\end{align*}with $\alpha_{ij},\beta_{i,j}\in \Z$, $i,j\in \{1,2\}$.

 We will now find a finite set of primes $F$  such that
$v_2\Lwedge w_2$ is  divisible by $p^k$ whenever  $v_1\Lwedge w_1$
is divisible by $p^k$, for every prime $p\notin F$.  Since the whole
process can be repeated  for $\phi^{-1}$, this will show that
$v_1\Lwedge w_1$ and $v_2\Lwedge w_2$ have the same type.

Since $\phi$ is an isomorphism, the matrix
\[M= \left(\begin{array}{ll}
\alpha_{11} & \alpha_{21}\\
\alpha_{12} & \alpha_{22}\\
\beta_{11} & \beta_{21}\\
\beta_{12} & \beta_{22}\end{array}\right)\] has rank two. At least
one of the following submatrices must then have rank 2 as well:
\[
M_1=\left(\begin{array}{ll} \alpha_{11} &\alpha_{21}\\
\alpha_{12}&\alpha_{22}\end{array}\right), \quad
M_2=\left(\begin{array}{ll} \beta_{11} &\beta_{21}\\
\beta_{12}&\beta_{22}\end{array}\right) \quad \mbox{ or } \quad
M_3=\left(\begin{array}{ll} \alpha_{11} &\alpha_{21}\\
\beta_{11}&\beta_{21}\end{array}\right).\]

Let $p$ be any prime not dividing $\det(M_1)$, $\det(M_2)$ or
$\det(M_3)$ and suppose $p^k$ divides  $v_1\Lwedge w_1$. By Lemma
\ref{det} there is an element $A=k_1v_1+k_2w_1\in \Gamma_1$
divisible by $p^k$ with either  $k_1$ or $k_2$ coprime with $p$.
Then
\begin{align} \phi(A,0)&=k_1\phi(v_1,0)+k_2 \phi(w_1,0) =\label{A0}\\
 &\Bigl((k_1 \alpha_{11}+k_2 \alpha_{21}) v_2+ (k_1 \alpha_{12}+k_2
\alpha_{22}) w_2\, ,\, (k_1 \beta_{11}+k_2 \beta_{21}) v_2+ (k_1
\beta_{12}+k_2 \beta_{22})w_2\Bigr)\in \Gamma_2\times \Gamma_2
 \notag
\end{align}

Suppose for instance that $M_1$ has rank 2. The only solution modulo
$p$ to the system
\[
\left\{ \begin{array}{lcl} \alpha_{11}x+\alpha_{21}y&=&0\\
\alpha_{12}x+\alpha_{22}y&=&0\end{array}\right.\] is then the
trivial one. The integers   $k_1$ and $k_2$ cannot therefore be  a
solution to the system (they are not both coprime with $p$). It
follows that one of the integers $k_1\alpha_{11}+k_2\alpha_{21}$ or
$\alpha_{12}k_{1}+\alpha_{22}k_2$ is not  a multiple of $p$.


If $M_2$ or $M_3$ have  rank two we argue exactly in the same way.
At the end  we find that at least one of the
$k_1\alpha_{1i}+k_2\alpha_{2i}$ or $k_1\beta_{1i}+k_2\beta_{2i}$ is
not a multiple of $p$ .

We know by \eqref{A0} that both $(k_1 \alpha_{11}+k_2 \alpha_{21})
v_2+ (k_1 \alpha_{12}+k_2 \alpha_{22})w_2$ and $(k_1 \beta_{11}+k_2
\beta_{21}) v_2+ (k_1 \beta_{12}+k_2 \beta_{22})w_2$ are divisible
by $p^k$ and we conclude with Lemma \ref{det} that $v_2\Lwedge w_2$
is divisible by $p^k$.
\end{proof}

\textbf{Proof of Theorem \ref{counter}} To see that
$K_1(C^\ast(\Delta_1))\cong K_1(C^\ast(\Delta_2))$, simply  put
together Lemma \ref{segprodext} and Lemma \ref{lem:desc}.

Since $\Gamma_1$ and $\Gamma_2$ are not isomorphic and finitely
generated Abelian groups have the cancellation property, $\Delta_1$
and $\Delta_2$ cannot be isomorphic, either.
\endproof
\begin{remark}
The argument of Lemma \ref{lem:desc} shows that
$K_0(C^\ast(\Delta))$ is (again) isomorphic to $\Z\oplus \Z \oplus
\Gamma \oplus \Gamma \oplus \Lwedge^2(\Gamma) \oplus
\Lwedge^2(\Gamma)$ (this time
 $K_0(C^\ast(\Delta))\cong \Lwedge^0\Delta
  \oplus \Lwedge^2\Delta \oplus \Lwedge^4 \Delta$ with
$\Lwedge^2\Delta\cong \Z \oplus \Gamma \oplus \Gamma \oplus
\Lwedge^2\Gamma$ and $\Lwedge^4 \Delta\cong \Lwedge^2 \Gamma$).

The group $C^\ast$-algebras $\DI$ and $\DII$ of Theorem
\ref{counter} have therefore the same $K$-theory.
\end{remark}
\section{Relating  $\UG$
and  $\Gamma$} \label{charac} This Section is devoted to evidence
what is the relation between  two $C^\ast$-algebras
$C^\ast(\Gamma_1)$ and $C^\ast(\Gamma_2)$ with topologically
isomorphic unitary groups. A result like Theorem \ref{index} cannot
be expected for general Abelian groups, as for instance all
countably infinite torsion groups have isometric $C^\ast$-algebras.
The right question to ask is obviously  whether group
$C^\ast$-algebras are determined by their unitary groups. Even if
this question also has a  negative answer,   two group
$C^\ast$-algebras $C^\ast(\Gamma_1)$ and $C^\ast(\Gamma_2)$ are
 strongly related when $\UGI$ and $\UGII$ are
topologically isomorphic as the contents of this Section show. Our
main tools here will be of topological nature and we shall regard
$\UG$ as $C(\widehat{\Gamma},\T)$.

We begin with a well-known observation. Denote by $C^0(X,\T)$ the
subgroup of $C(X,\T)$ consisting of all nullhomotopic maps, that is,
$C^0(X, \T)$ is  the connected component of the identity of
$C(X,\T)$. Let also  $\pi^1(X)$ denote the quotient
$C(X,\T)/C^0(X,\T)$, also known as the first cohomotopy group of $X$
and often denoted as  $[X,\T]$. It is well known that $C^0(X,\T)$
coincides with the group of functions that factor through $\R$, that
is, $C^0(X,\T)$ is the range of the exponential map $\exp:C(X,\R)\to
C(X,\T)$.
\begin{lemma}[Section 3 of \cite{pest95}, see page 405 of \cite{galihern99fo} for this form]\label{basicc0}
If $X$ is a compact Hausdorff space, the structure of $C(X,\T)$ is
described by the following exact sequence \[ 0\to C(X,\Z)\to C(X,\R)
\to  C^0(X,\T) \to C(X,\T)\to \pi^1(X).\] In addition $C^0(X,\T)$ is
open and splits, i.e., $C(X,\T) \cong C^0(X,\T)\oplus \pi^1(X)$.
\end{lemma}
Our second observation is that, as far as group $C^\ast$-algebras
are concerned, all discrete Abelian groups have a splitting torsion
subgroup.
\begin{theorem}[Corollary 10.38 \cite{hoffmorr}]\label{split}
The connected component  $G_0$ of a compact group $G$, splits
topologically, i.e, $G$ is homeomorphic to $G_0\times G/G_0$.
\end{theorem}
The character  group of  a  countable discrete group $\Gamma$ is a
compact metrizable  group $\widehat{\Gamma}$ and the set of
characters that vanish on its torsion group, $t\Gamma$,  coincides
with the connected component  of $\widehat{\Gamma}$, in symbols
$t\Gamma^\perp=\widehat{\Gamma}_0$. Further, the duality between
discrete Abelian and compact Abelian groups   identifies
$\widehat{t\Gamma}$ with the quotient
$\widehat{\Gamma}/\widehat{\Gamma}_0$. It follows therefore from
Theorem \ref{split} that
\begin{equation}
\widehat{\Gamma} \sim \widehat{t\Gamma}\times
(t\Gamma)^{\perp}\label{eq-torsion2}
\end{equation}
and, hence, that  $C^\ast(\Gamma)$ is isometric to $C^\ast(t\Gamma
\oplus \Gamma/t\Gamma)$.

We now turn our attention to groups with splitting connected
component.
\subsection{The structure of unitary groups of certain commutative $C^\ast$-algebras}
We begin by noting that the additive structure of a commutative
$C^\ast$-algebra contains very little information on the algebra.
This fact will be useful in classifying unitary groups.
 \begin{theorem}[Milutin, see for instance Theorem III.D.18 of \cite{wojt}]\label{milutin}
If $K_1$ and $K_2$ are uncountable, compact metric spaces, then the
Banach  spaces $C(K_1,\C)$ and $C(K_2,\C)$ are topologically
isomorphic.
\end{theorem}
\begin{lemma}\label{estrcdete}
Let $K$ and $D$ be  compact  topological spaces, $K$ connected and
$D$ totally disconnected. The following topological isomorphism then
holds: \begin{equation} \label{eq:estrcdete} C(K\times D,\T)\cong
C(K\times D ,\R) \times C(D,\T)\times \oplus_{w(D)}
\pi^1(K),\end{equation} where $w(D)$  denotes the topological weight
of $D$ \footnote{By the topological weight of a topological space
$X$ we mean, as usual, the least cardinal number of a basis of open
sets of $X$.}
\end{lemma}
\begin{proof}
We first observe that  $C(K\times D,\T)$ is topologically isomorphic
to $C(D,C(K,\T))$.
 From Lemma  \ref{basicc0} we deduce that
\begin{equation}\label{eqcs} C(K\times D,\T) \cong C(D,C^0(K,\T))\times C(D,\pi^1(K)).
\end{equation}
There is a topological isomorphism from the  Banach space $C(K,\R)$
onto the Banach space $C_\bullet(K,\R)$ of functions sending $0$ to
0. It is now easy to check that the mapping $(f,t)\mapsto t\cdot
\exp( f)$
 identifies $C_\bullet(K,\R)\times \T$ with $C^0(K,\T)$ and hence
\[ C^0(K,\T) \cong C(K,\R) \times \T.\]
Along with \eqref{eqcs} we  obtain
\[ C(K\times D,\T)\cong C(D,C(K,\R)\times \T) \times C(D,\pi^1(K))=C(D\times K,\R)\times
 C(D,\T)\times C(D,\pi^1(K)) .\]
Now $\pi^1(K)$ is a discrete group and each element of
$C(D,\pi^1(K))$  determines an  open and closed subset of $D$. An
analysis identical to that of \cite{edaohtakamo} for $C(X,\Z)$ then
yields
\[ C(D,\pi^1(K))\cong \oplus_{w(D)} \pi^1(K),\]
and the proof follows.
\end{proof}

The following lemma can be found as  Exercise E.8.14  in
\cite{hoffmorr}.
\begin{lemma}\label{conntd}
If $D$ is a totally disconnected compact space, $C(D,\T)=C^0(D,\T)$
and $C(D,\T)$  is connected.
\end{lemma}

\begin{theorem}\label{esp-teocomp} Let  $K_i$, $i=1,2$,  be two compact connected metrizable
spaces and let  $D_i$, $i=1,2$, be two compact totally disconnected
metrizable spaces. Defining  $X=K_1\times D_1$ and $Y=K_2\times
D_2$,the following assertions are equivalent.
\begin{enumerate}
\item $C(X,\torus)\cong C(Y,\torus)$.
\item
\begin{itemize}
\item[(a)] $\bigoplus_{w(D_1)}\pi^{1}(K_1)\cong \bigoplus_{w(D_2)}\pi^{1}(K_2)$,
where $w(D_1)$ and $w(D_2)$ are the topological weights of $D_1$ and
$D_2$, respectively, and

\item[(b)] $C(D_1,\torus)\cong C(D_2,\torus)$.
\end{itemize}\end{enumerate}
\end{theorem}
\begin{proof}
It is obvious from Theorem \ref{milutin} (observe that $K_i\times
D_i$ is uncountable as soon as $K_i$ is nontrivial) and Lemma
\ref{estrcdete} that (2) implies (1).

We now use the decomposition of Lemma \ref{estrcdete} to deduce (2)
from (1). Assertion (a) follows from factoring out  connected
components in \eqref{eq:estrcdete} (note that  $C(K_i\times
D_i,\R)\times C(D_i,\T)$ is connected, use Lemma \ref{conntd} for
$C(D_i,\T)$). The connected components $C(K_1\times D_1,\R)\times
C(D_1,\T)$ and $C(K_2\times D_2,\R)\times C(D_2,\T)$ will be
topologically isomorphic as well. Let $H\colon C(K_1\times
D_1,\R)\times C(D_1,\T)\to C(K_2\times D_2,\R)\times C(D_2,\T)$
denote this isomorphism.

Consider now the homomorphism $\widehat{H} \colon C(K_2\times
D_2,\R)^{\;\widehat{\;}\;}\times C(D_2,\T)^{\;\widehat{\;}\;}\to
C(K_1\times D_1,\R)^{\;\widehat{\;}\;}\times
C(D_1,\T)^{\;\widehat{\;}\;} $ that results from dualizing $H$.

When  $D$ is a totally disconnected compact group, the only
continuous characters of   $C(D,\T)$ are linear combinations with
coefficients in $\Z$ of evaluations of elements of $D$, i.e., the
group $C(D,\T)^{\;\widehat{\;}\;} $ is isomorphic to the free
Abelian  group $A(D)$  on $D$ \cite{pest95} (see \cite{galihern99fo}
for more on the duality between $C(X,\T)$ and $A(X)$ based on the
exact sequence in Lemma \ref{basicc0})).

There is on the other hand  a well-known isomorphism between
$C(K_1\times D_1,\R)^{\;\widehat{\;}\;}$ and the additive group of
the vector space of all continuous linear functionals on
$C(K_1\times D_1,\R)$. The group  $C(K_1\times
D_1,\R)^{\;\widehat{\;}\;}$ is therefore a divisible.

Since free Abelian  groups, such as  $A(D_i)$, do not contain any
divisible subgroup, $\widehat{H}(C(K_1\times
D_1,\R)^{\;\widehat{\;}\;}$  must equal $C(K_2\times
D_1,\R)^{\;\widehat{\;}\;}$. We deduce thus, taking quotients, that
$C(D_1,\T)$ and $C(D_2,\T)$ are topologically isomorphic.
\end{proof}
\subsection{The group case}
We now specialize the results in the previous paragraphs for the
case of a compact Abelian group.

When $T$ is a torsion discrete Abelian group, $\widehat{T}$ is a
compact totally disconnected group and hence  homeomorphic to the
Cantor
 set. The group $C^\ast$-algebras of all  countably infinite  torsion Abelian groups
 will therefore be isometric. These facts are summarized in the
 following lemma.
\begin{lemma}\label{tor}
Let $T_1$ and $T_2$ be countable torsion discrete Abelian groups.
Then the following assertions are equivalent:
\begin{enumerate}
\item The group $C^{\ast}$-algebras $C^{\ast}(T_1)$ and $C^{\ast}(T_2)$ are
isomorphic as $C^{\ast}$-algebras.
\item The unitary groups of $C^{\ast}(T_1)$ and $C^{\ast}(T_2)$ are topologically
isomorphic.
\item The compact  groups $\widehat{T_1}$ and $\widehat{T_2}$ are
homeomorphic.
\item The groups $T_1$ and $T_2$ have the same cardinal. 
\end{enumerate}\label{cedete-discr}
\end{lemma}

\noindent Hence, the main result asserts:
\begin{theorem}
Let $\Gamma_1$ and $\Gamma_2$ be countable  discrete Abelian groups.
The following are equivalent:
\begin{enumerate}
\item \label{1}The unitary groups of $C^{\ast}(\Gamma_1)$ and $C^{\ast}(\Gamma_2)$ are
topologically isomorphic.

\item \label{2}
$|t\Gamma_1|=|t\Gamma_2|=\alpha$  and
\[ \bigoplus_{\alpha}\frac{\Gamma_1}{t\Gamma_1}\cong
\bigoplus_{\alpha}\frac{\Gamma_2}{t\Gamma_2}.\]
\end{enumerate}
\label{teo-grupos-disc-num}
\end{theorem}
\begin{proof}
Using  the homeomorphic identification in \eqref{eq-torsion2}, page
\pageref{eq-torsion2}, and Lemma \ref{estrcdete} we have:
\begin{equation}\label{basrel}
\mathcal{U}(C^{\ast}(\Gamma_i))\cong C( \widehat{ t\Gamma_i}\times
(t\Gamma_i)^{\perp},\T\bigr)\cong  C(\widehat{t\Gamma_i}\times
(t\Gamma_i)^\perp,\R)\times C(\widehat{t\Gamma_i},\T)\times
\bigoplus_{w(\widehat{t\Gamma_i})}
\pi^1((t\Gamma_i)^\perp),\end{equation} where $(t\Gamma_i)^{\perp}$
are compact connected and $ \widehat{t\Gamma_i}$ are compact totally
disconnected Abelian groups.

Suppose first that $\mathcal{U}(C^{\ast}(\Gamma_1))$ and
$\mathcal{U}(C^{\ast}(\Gamma_2))$ are topologically isomorphic. By
Theorem \ref{esp-teocomp}, $C( \widehat{t\Gamma_1},\T)$ is
topologically isomorphic to $C(\widehat{ t\Gamma_2},\T)$. It follows
from  Lemma \ref{tor} that $\widehat{t\Gamma_1}$ and
$\widehat{t\Gamma_2}$ are homeomorphic. Let
$\alpha=w(\widehat{t\Gamma_1})$. By statement (a) of  Theorem
\ref{esp-teocomp},
\begin{displaymath}
\bigoplus_{\alpha} \pi^{1}((t\Gamma_1)^{\perp})\cong
\bigoplus_{\alpha}
\pi^{1}((t\Gamma_2)^{\perp}),
\end{displaymath}
Now $\pi^1(t\Gamma_i^\perp)$ is isomorphic  by Theorem \ref{index}
to the torsion-free  group $\Gamma_i /t(\Gamma_i)$. The above
isomorphism thus becomes
\begin{equation}
\bigoplus_{\alpha} \left (\frac{\Gamma_1}{t\Gamma_1}\right)\cong
\bigoplus_{\alpha}
\left(\frac{\Gamma_2}{t\Gamma_2}\right)\label{eq-gamma}
\end{equation}
and we are done.

Suppose conversely that assertion \eqref{2} holds. We have then from
Lemma \ref{tor} that $C(\widehat{t\Gamma_1},\T)$ and
$C(\widehat{t\Gamma_1},\T)$ are topologically isomorphic.

On the other hand, the isomorphism
$\bigoplus_{\alpha}\frac{\Gamma_1}{t\Gamma_1}\cong
\bigoplus_{\alpha}\frac{\Gamma_2}{t\Gamma_2}$ implies, by way of
Theorem \ref{index}, that $\oplus_\alpha\pi^1((t\Gamma_1)^\perp)$ is
isomorphic to $\oplus_\alpha\pi^1((t\Gamma_2)^\perp)$.

It follows then from Theorem \ref{esp-teocomp} that
$C(\widehat{\Gamma_1},\T)$ and $C(\widehat{\Gamma_2},\T)$, that is
$\mathcal{U}(C^{\ast}(\Gamma_1))$ and
$\mathcal{U}(C^{\ast}(\Gamma_2))$, are topologically isomorphic.
\end{proof}
\section{Concluding remarks}\label{concl}
Theorem \ref{index} shows how strongly the topological group
structure of $\mathcal{U}(\A)$ may  happen to determine   a
$C^\ast$-algebra $\A$. Theorem \ref{teo-grupos-disc-num} then
precises the amount of
 information on $\A$ that is encoded in $\mathcal{U}(\A)$, for the case of a commutative group
$C^\ast$-algebra. This reveals some limitations on the strength of
$\mathcal{U}(\A)$ as an invariant of $\A$ that will be made concrete
in this Section.

$\,$From Theorem \ref{index} and Lemma  \ref{tor} we have that
$C^{\ast}(\Gamma)$ is completely determined  by its unitary group
when $\Gamma$ is either torsion-free or a torsion group. This is not
the case if $\Gamma$ is a mixed group.
 \begin{example}\label{noinso}
Two nonisometric Abelian group $C^\ast$-algebras with topologically
isomorphic unitary groups. \end{example}
\begin{proof} Let $\Gamma_1$ and $\Gamma_2$ be the groups in
Theorem \ref{teo-fuchs}. Define $\Delta_i=\Gamma_i\oplus \Z_2$.
Identifying as usual $C(\Delta_i,\T)$ with $\UDIi$ and applying
Lemma \ref{estrcdete}, we have that
\[\UDIi\cong C(\Delta_i,\R)\times \T^2 \times
(\Gamma_i\oplus \Gamma_i).\] The election of $\Gamma_i$  and
Milutin's theorem show that $\UDI$ is topologically isomorphic to
$\UDII$.

The algebras $C^\ast(\Delta_1)$ and $C^\ast(\Delta_2)$ are not
isometric, since their spectra, $\widehat{\Gamma_1}\times \Z_2$ and
$\widehat{\Gamma_2}\times \Z_2$, are not homeomorphic (their
connected components are not homeomorphic).
 \end{proof}
This  example also shows that simple "duplications" of torsion-free
groups are not determined by the unitary groups of their
$C^\ast$-algebras:
\begin{example}\label{rank}
Two nonisomorphic torsion-free Abelian groups $\Gamma_1$ and
$\Gamma_2$ such that $\U(C^\ast(\Gamma_1\oplus \Z_2))$ and
$\U(C^\ast(\Gamma_2\oplus \Z_2))$ are topologically isomorphic.
\end{example}
Finally,
\begin{example}
Two Abelian groups $\Gamma_1$ and $\Gamma_2$ of different
torsion-free rank with $\UGI$ topologically isomorphic to
$\UGII$.\label{ejemplo}
\end{example}
\begin{proof}
Let ${\displaystyle \Gamma_1= \Z \oplus (\oplus_\omega \Z_2)}$ and
 ${\displaystyle \Gamma_2= \left(\Z\oplus \Z \right) \oplus (\oplus_\omega
 \Z_2)}$. The argument now is as in Example \ref{noinso}.
 \end{proof}
In the above example one can obviously replace $\Gamma_2$ by
$\left(\oplus_\omega  \Z \right) \oplus (\oplus_\omega
 \Z_2)$ and  have an example of two Abelian groups  with $\UGI$ topologically isomorphic to $\UGII$
 while the torsion-free rank of one of them is finite and the
 torsion-free rank of the other is infinite.
 \subsection{Invariants}\label{comp}
 The unitary group $\UG$ is an invariant of the group $\CG$, and as
 such can be compared with other well known unitary-related
 invariants, like for instance $K_1(\CG)$.
We can also mention here related   work of  Hofmann and Morris  on
free
 compact Abelian groups \cite{hoffmorr86}. This is  part of  a more general project of
 attaching a compact topological group $FC(X)$ to every compact
 Hausdorff space $X$. The  free compact Abelian group  on $X$ is
 constructed as
 the character group of  the \emph{discrete} group $C(X,\T)_d$. For an Abelian  group
 $\Gamma$, this process
 produces an  invariant of
 $\CG$,  namely the group
 $\UG_d$  equipped with the discrete topology. The
 character group of $\UG_d$ is precisely the free compact Abelian group on
 $\widehat{\Gamma}$. Being the same object but with no topology,
 this
 invariant is weaker  than  $\UG$. It is easy to see that it is indeed strictly weaker,
 simply take
  $\Gamma_1=\Q$ and
 $\Gamma_2=\oplus_\omega \Q$. In general there is a copy of the
 free Abelian group generated by $X$, densely embedded in $FC(X)$,
$ FC(X)$ is, actually (a realization of) the Bohr compactification
of the free Abelian topological group on $X$ (see
\cite{galihern99fo} for detailed references on free Abelian
topological groups and their duality properties). Since  two
topological spaces with topological isomorphic free Abelian
topological groups must have the same covering dimension
\cite{pest82}, Example \ref{rank} is somewhat unexpected.

The comparison with $K_1(\UG)$ is richer. As we saw in Section
\ref{sec:counter}, the group algebras $\CCG{1}$ and $\CCG{2}$ of two
nonisomorphic torsion-free Abelian groups $\Gamma_1$ and $\Gamma_2$
can have isomorphic $K_1$-groups, while their unitary groups must be
topologically isomorphic by Theorem \ref{index}. The opposite
direction  does not work either. We find next  two discrete groups
whose group $C^{\ast}$-algebras have isomorphic unitary groups while
their $K_1$-groups fail to be so. We first see that from Theorem
\ref{split} and with a simple application of the K\"unneth theorem,
the $K_1$-group of a group $C^\ast$-algebra depends exclusively on
its torsion-free component.
\begin{lemma}\label{kunn}
Let $\Gamma$ be an Abelian discrete group. Then
\[
K_1(C^{\ast}(\Gamma))\cong K_1(C^{\ast}(\Gamma/t\Gamma))
\]
\end{lemma}
\begin{proof}
From Theorem \ref{split}, $\widehat{\Gamma}$ is homeomorphic to
$\widehat{\Gamma}/\widehat{\Gamma}_0\times \widehat{\Gamma}_0$,
where $\widehat{\Gamma}/\widehat{\Gamma}_0\cong \widehat{t\Gamma}$
and $\widehat{\Gamma}_0\cong  t\Gamma^{\perp} \cong
\widehat{\Gamma/t\Gamma}$. Therefore,
\begin{equation}
C^{\ast}(\Gamma)\cong C^{\ast}(t\Gamma)\otimes
C^{\ast}(\Gamma/t\Gamma).\label{eq-kunn}
\end{equation}
Applying the Künneth formula to (\ref{eq-kunn}), we obtain,
\begin{eqnarray*}
K_1(C^{\ast}(\Gamma))&\cong & K_1(C^{\ast}(t\Gamma)\otimes
C^{\ast}(\Gamma/t\Gamma))\\
&\cong & K_0(C^{\ast}(t\Gamma))\otimes
K_1(C^{\ast}(\Gamma/t\Gamma))\oplus K_1(C^{\ast}(t\Gamma))\otimes
K_0(C^{\ast}(\Gamma/t\Gamma))\\
&\cong & \mathbb{Z}\otimes K_1(C^{\ast}(\Gamma/t\Gamma))\cong
K_1(C^{\ast}(\Gamma/t\Gamma)),
\end{eqnarray*}
since $K_0(C(D))=\mathbb{Z}$ and $K_1(C(D))=0$ for a infinite
totally disconnected compact group $D$.
\end{proof}

\begin{example}
Two Abelian groups $\Gamma_1$ and $\Gamma_2$ whose group
$C^{\ast}$-algebras have topologically isomorphic unitary groups,
whereas their $K_1$-groups are nonisomorphic.
\end{example}
\begin{proof}
Take $\Gamma_1$ and $\Gamma_2$ from Example \ref{ejemplo}. Applying
Lemma \ref{kunn} and Lemma \ref{kgrupo-ext}, we have that
\[
K_1(\CCG{1})\cong K_1(C^{\ast}(\mathbb{Z}))\cong\mathbb{Z}\quad
\text{ and } \quad K_1(\CCG{2})\cong K_1(C^{\ast}(\mathbb{Z\oplus
\Z}))\cong \Z \oplus\Z.
\]
The topological groups $\mathcal{U}(\CCG{1})$ and
$\mathcal{U}(\CCG{2})$ are topologically isomorphic as  was proved
in Example \ref{ejemplo}.
\end{proof}

As a consequence,  we see  that none of the invariants
$\mathcal{U}(C^{\ast}(\Gamma))$ and  $K_1(C^{\ast}(\Gamma))$, of  a
group algebra $C^{\ast}(\Gamma)$ is stronger than the other. The
groups in Theorem  \ref{counter} also show that two nonisometric
(Abelian) $C^\ast$-algebras can have topologically  isomorphic
unitary groups \emph{and} isomorphic $K_1$-groups. Take $\Phi_i
=\Delta_i\times \Z_2$ with $\Delta_i$ defined as in Theorem
\ref{counter}. The same argument of Example \ref{noinso} shows that
$\mathcal{U}(C^\ast(\Delta_i))\cong C(\Phi_i,\R)\times \T^2\times
\Delta_i \times \Delta_i$  and,  hence, that
$\mathcal{U}(C^\ast(\Phi_1))\cong \mathcal{U}(C^\ast(\Phi_2))$. To
see that $K_1(C^\ast(\Phi_1))\cong K_1(C^\ast(\Phi_2))$ simply note
that, by Lemma \ref{kunn}, $K_1(C^\ast(\Phi_i))\cong
K_1(C^\ast(\Delta_i))$ and that $K_1(C^\ast(\Delta_1))\cong
K_1(C^\ast(\Delta_2))$ by Theorem \ref{counter}.

\def\cprime{$'$} \def\cprime{$'$}
  \def\polhk#1{\setbox0=\hbox{#1}{\ooalign{\hidewidth
  \lower1.5ex\hbox{`}\hidewidth\crcr\unhbox0}}}
  \def\polhk#1{\setbox0=\hbox{#1}{\ooalign{\hidewidth
  \lower1.5ex\hbox{`}\hidewidth\crcr\unhbox0}}} \def\cprime{$'$}

\end{document}